\RequirePackage{fix-cm}
\documentclass[smallextended]{svjour3_2}       
\smartqed  
\usepackage{graphicx}
 \usepackage{mathptmx}      
 \journalname{ }

\usepackage{amssymb,amsmath}
\usepackage{ntheorem}
\usepackage{xspace}
\usepackage{euscript} 
\usepackage{bussproofs}
\usepackage{cases}
\usepackage{stmaryrd}

\makeatletter \let\cl@chapter\relax \makeatother 
\usepackage{cleveref}

\usepackage{esvect}

\DeclareMathAlphabet{\mathcal}{OMS}{cmsy}{m}{n}

\newtheorem{thm}{Theorem}[section]
\numberwithin{thm}{section}
\newtheorem{lem}[thm]{Lemma}
\newtheorem{dfn}[thm]{Definition}
\newtheorem{rmk}[thm]{Remark}

\Crefname{dfn}{Definition}{Definitions}
\Crefname{lem}{Lemma}{Lemmas}
\Crefname{thm}{Theorem}{Theorems}
\Crefname{rmk}{Remark}{Remarks}

\newcommand{\IPC}{\ensuremath{\mathsf{IPC}}\xspace}
\newcommand{\CPC}{\ensuremath{\mathsf{CPC}}\xspace}

\newcommand{\PD}{\ensuremath{\mathbf{PD}}\xspace}
\newcommand{\PT}{\ensuremath{\mathbf{PT}}\xspace}
\newcommand{\PID}{\ensuremath{\mathbf{PID}}\xspace}

\newcommand{\PTo}{\ensuremath{\mathsf{PT}}\xspace}
\newcommand{\PIDo}{\ensuremath{\mathsf{PID}}\xspace}
\newcommand{\PDo}{\ensuremath{\mathsf{PD}}\xspace}

\newcommand{\Inql}{\ensuremath{\mathsf{InqL}}\xspace}
\newcommand{\KP}{\ensuremath{\mathsf{KP}}\xspace}
\newcommand{\ND}{\ensuremath{\mathsf{ND}}\xspace}

\newcommand{\ML}{\ensuremath{\mathsf{ML}}\xspace}

\newcommand{\LLn}{\ensuremath{{\mathsf{L}^\neg}}\xspace}
\newcommand{\LL}{\ensuremath{\mathsf{L}}\xspace}

\newcommand{\dep}{\ensuremath{\mathop{=\!}}\xspace}

\newcommand{\sor}{\ensuremath{\otimes}\xspace}

\newcommand{\bor}{\ensuremath{\vee}\xspace}

\newcommand{\bigsor}{\bigotimes}
\newcommand{\bigbor}{\bigvee}

\newcommand{\adm}{\mid\!\!\sim}

\newcommand{\Su}{\ensuremath{\mathcal{S}}\xspace}
\newcommand{\Fl}{\ensuremath{\mathcal{F}}\xspace}
\newcommand{\St}{\ensuremath{\mathcal{ST}}\xspace}

\newcommand{\eqv}[1]{\sim_{#1}} 
\newcommand{\depp}[2]{\dep(\vv{#1},{#2})} 
\newcommand{\lang}[1]{${\EuScript L}_{#1}$}

\newcommand{\cnt}{\ensuremath{\divideontimes}\xspace}
\newcommand{\imp}{\rightarrow} 
\newcommand{\A}{\forall} 
\newcommand{\Ifff}{\ensuremath{\Leftrightarrow}\xspace}

\newcommand{\sig}{\ensuremath{\sigma}}
\renewcommand{\phi}{\varphi}
\newcommand{\Ga}{\Gamma}

\begin{document}

\title{Structural completeness in propositional logics of dependence\thanks{Support by the Netherlands Organisation for Scientific 
Research under grant 639.032.918 is gratefully acknowledged.}
}


\author{Rosalie Iemhoff         \and
        Fan Yang 
}


\institute{Rosalie Iemhoff \at
              Department of Philosophy and Religious Studies, Utrecht University\\
               Janskerkhof 13, 3512 BL Utrecht, The Netherlands \\
              \email{r.iemhoff@uu.nl}           
           \and
           Fan Yang \at
              Department of Philosophy and Religious Studies, Utrecht University\\
               Janskerkhof 13, 3512 BL Utrecht, The Netherlands \\
              \email{fan.yang.c@gmail.com} 
}


\maketitle

\begin{abstract}
In this paper we prove that three of the main propositional logics of dependence (including propositional dependence logic and  inquisitive logic), none of which is structural, are structurally complete with respect to a class of substitutions under which the logics are closed. We obtain an  analogues result with respect to stable substitutions, for the negative variants of some well-known intermediate logics, which are intermediate theories that are closely related to inquisitive logic. 

\end{abstract}

\section{Introduction}
\label{intro}
In recent years there have appeared many results on admissible rules in logics. The diversity of the results show that the properties of admissibility vary from logic to logic, and the complexity of some of the results show that describing these rules is not always an easy matter. 
The admissible rules of a logic are the rules under which the logic is closed, meaning that one could add them to the logic without obtaining new theorems. Since adding a derivable rule to a logic cannot alter that what can be derived, derivable rules are always admissible, thus showing that the notion of admissibility is a natural extension of the notion of derivability.    

For {\em structural logics}, which means logics that are closed under uniform substitution, a rule is admissible if every substitution that unifies the premiss, unifies the conclusion, where a substitution $\sig$ unifies a formula $\phi$ in a logic if $\sig\phi$ is derivable in the logic. Until now, most logics for which the admissibility relation has been studied are structural. Main examples are classical and intuitionistic propositional logic and certain modal logics such as {\sf K}, {\sf K4}, and {\sf S4}. Except for classical logic, all these logics have nonderivable admissible rules and their admissibility relations are decidable and have concise axiomatizations   \cite{ghilardi99,iemhoff01,jerabek05,roziere92,rybakov97}. 
In recent years, admissibility has been studied for a plethora of other logics as well. However, logics that are not structural, have received less attention. In order to obtain a meaningful notion of admissibility for such a logic one first has to isolate a set of substitutions, as large as one thinks possible, under which the logic is closed. Admissibility can then be studied with respect to this class of substitutions. 

In this paper we show that three of the main propositional logics of dependence, none of which is structural, are structurally complete with respect to the class of \emph{flat} substitutions. We obtain an  analogues result, but then with respect to \emph{stable} substitutions, for the negative variants of some well-known intermediate logics, which are intermediate theories that are closely related to one of the logics of dependence. As a byproduct we develop an extension of the usual logics of dependence in which the use of negation and the dependence atom is not restricted to propositional variables, but to the much larger class of flat formulas instead. 

We think the interest in these results lies in the fact that logics of dependence, to be described below, are  versatile and widely applicable nonclassical logics. And knowing that many nonclassical logics have nontrivial adimissible rules, establishing that in these logics all rules that are admissible (with respect to flat substitutions) are derivable, provides a useful insight in the logics. Moreover, 
these results provide one of the first examples of natural nonstructural logics for which admissibility is  studied. A paper in which various admissibility relations of nonstructural logics, the same logics that we treat in Theorem~\ref{strcpl_neg_int}, have been studied is \cite{Migliolietc89}, but the results are different from the ones obtained here. 

\emph{Dependence logic} is a new logical formalism that characterizes the notion of ``dependence'' in social and natural sciences. 
First-order dependence logic was introduced by V\"{a}\"{a}n\"{a}nen \cite{Van07dl} as a development of \emph{Henkin quantifier} \cite{henkin61}  and \emph{independence-friendly logic}  \cite{HintikkaSandu1989}. Recently, propositional dependence logic (\PDo) was studied and axiomatized in \cite{SanoVirtema2014,VY_PD}. With a different motivation, Ciardelli and Roelofsen \cite{InquiLog} introduced and axiomatized \emph{propositional inquisitive logic} (\Inql), which can be regarded as a natural variant of propositional dependence logic. Both \PDo and \Inql are fragments of \emph{propositional downwards closed team logic} (\PTo), which was studied in \cite{VY_PD} and essentially also in \cite{Ciardelli2015}. Dependency relations are characterized in these propositional logics of dependence by a new type of atoms $\dep(\vv{p},q)$, called \emph{dependence atoms}. Intuitively, the atom specifies that \emph{the proposition $q$ depends completely on the propositions $\vv{p}$}. The semantics of these logics is called \emph{team semantics}, introduced by Hodges \cite{Hodges1997a,Hodges1997b}. The basic idea of this new semantics is that properties of dependence cannot be manifested in \emph{single} valuations, therefore unlike the case of classical propositional logic, formulas in  propositional logics of dependence are evaluated on \emph{sets} of valuations (called \emph{teams}) instead.


The three logics \PDo, \Inql and \PTo are of particular interest, because they are all \emph{expressively complete}, in the sense that they characterize all downwards closed nonempty collections of teams. 
As a result of the feature of team semantics, the sets of theorems of these logics are closed under flat substitutions, but not closed under \emph{uniform substitution}. As mentioned above, in this paper we prove that the three logics are structurally complete with respect to flat substitutions. 

In the study of admissible rules there is a technical detail that needs to be addressed. In \PDo and \PTo, negation and the dependence operator can only be applied to atoms. Therefore, the only substitutions under which these logics are closed are \emph{renamings}, substitutions that replace atoms by atoms. However, these logics can be conservativily extended to logics that are closed under flat substitutions. These extensions, \PD and \PT, are closed under flat substitutions, and for these logics, as well as for \Inql, the notion of admissibility with respect to flat substitutions is shown to be equal to derivability (Theorem~\ref{maintheorem}). 


There is a close connection between inquisitive logic and certain intermediate logics. 
The set of theorems of the former equals the negative variant of Kreisel-Putnam logic (\KP), which is equal to the negative variant of  Medvedev logic (\ML). It is open whether  \KP is structurally complete, whereas \ML is known to be structurally complete but not hereditarily structurally complete. An interesting corollary we obtain in this paper is that the negative variants of both \ML and \KP are hereditarily structurally complete with respect to negative substitutions.

\section{Logics of dependence}

\subsection{Syntax and semantics}

We first define propositional downwards closed team logic. All of the logics of dependence we consider in the paper are fragments of propositional downwards closed team logic.
\begin{dfn}\label{syntax_pd}
Let $p,q,\vv{p}=p_1,\dots,p_k$ be propositional variables. Well-formed formulas of \emph{propositional downwards closed team logic} (\PTo) are given by the following grammar:
\[
    \phi::= \,p\mid \neg p\mid\bot\mid \top\mid\dep(\vv{p},q)\mid\phi\wedge\phi\mid\phi\sor\phi\mid \phi\vee\phi\mid \phi\to\phi.
\] 
\end{dfn}

We call the formulas $p$, $\neg p$, $\bot$ and $\top$ \emph{propositional atoms}. The formula $\dep(\vv{p},q)$ is called a \emph{dependence atom}, and it shall be read as ``$q$ depends on $\vv{p}$\,''. The connective $\sor$ is called \emph{tensor} (disjunction), and the connectives $\vee$ and $\to$  are called \emph{intuitionistic disjunction} and \emph{intuitionistic implication}, respectively. The formula $\phi\to\bot$ is abbreviated as $\neg\phi$, and the team semantics to be given guarantees that the formula $\neg p$ and $p\to \bot$ are semantically equivalent.

Fragments of \PTo formed by certain sets of atoms and connectives in the standard way are called \emph{(propositional) logics of dependence}. The following table defines the syntax of the other logics of dependence we consider in this paper. 


\begin{center}
\begin{tabular}{|l|c|c|}
\multicolumn{1}{c}{\textbf{Logic}}&\multicolumn{1}{c}{\textbf{Atoms}}&\multicolumn{1}{c}{\textbf{Connectives}}\\\hline\hline
Propositional dependence logic (\PDo)&$p,\neg p,\bot,\top,\dep(\vv{p},q)$&$\wedge,\sor$\\\hline
Propositional inquisitive logic (\Inql)&$p,\bot,\top$&$\wedge,\vee,\to$\\\hline
\end{tabular}
\end{center}



Given any of the three logics $\LL \in \{\PDo,  \Inql, \PTo\}$, let \lang{\LL} denote the language of \LL. We say that a formula $\phi$ {\em is in} \lang{\LL} if all symbols in $\phi$ belong to \lang{\LL}. Clearly, \lang{\Inql} is the same as the language of intuitionistic propositional logic or intermediate logics. We will discuss the connection between \Inql and intermediate logics in the sequel. Note that 
formulas in \lang{\PDo} are assumed to be in \emph{(strict) negation normal form}, in the sense that negation is allowed only in front of propositional variables and dependence atoms can not be negated. We will revisit the issue about negation in Section 3.1.




For the semantics, propositional logics of dependence adopt \emph{team semantics}. A \emph{team} is a set of valuations, i.e., a set of functions $v:\rm{Prop}\to\{0,1\}$, where \rm{Prop} is the set of all propositional variables. 

\begin{dfn}\label{TS_PD}
We inductively define the notion of a formula $\phi$ in \lang{\PTo}  being \emph{true} on a team $X$, denoted by $X\models\phi$, as follows:
\begin{itemize}
\item $X\models p$ iff
for all $v\in X$, $v(p)=1$;
  \item $X\models\neg p$  iff
for all $v\in X$, $v(p)=0$;
  \item $X\models\bot$ iff $X=\emptyset$;
  \item $X \models \top$ for all teams $X$;
  \item $X\models\,\dep(\vv{p},q)$ iff for all $v,v'\in X$: $v(\vv{p})=s'(\vv{p})~\Longrightarrow ~v(q)=v'(q)$;
  \item $X\models\phi\wedge\psi$ iff $X\models\phi$ and
  $X\models\psi$;
  \item $X\models\phi\sor\psi$ iff there exist teams $Y,Z\subseteq X$ with $X=Y\cup Z$ such that 
  $Y\models\phi$ and $Z\models\psi$;
  \item $X\models \phi\bor\psi$ iff $X\models \phi$ or $X\models\psi$;
  \item $X\models \phi\to\psi$ iff for any team $Y\subseteq X$: 
  \(Y\models \phi\,\Longrightarrow \,Y\models\psi.\)
\end{itemize}
\end{dfn}

 

If $X\models\phi$ holds for all teams $X$, then we say that $\phi$ is \emph{valid}, denoted by $\models\phi$. For a finite set of formulas $\Gamma$, we write $\Gamma\models\phi$ and say that $\phi$ is a \emph{logical consequence} of $\Gamma$ if $X\models \bigwedge\Gamma\Longrightarrow X\models\phi$ holds for all teams $X$. In case $\Gamma=\{\phi\}$, we write simply $\phi\models\psi$ instead of $\{\phi\}\models\psi$. If $\phi\models\psi$ and $\psi\models\phi$, then we write $\phi\equiv\psi$ and say that $\phi$ and $\psi$ are \emph{semantically equivalent}. Two logics of dependence $\LL_1$ and $\LL_2$ are said to have the \emph{same expressive power} if for every $\LL_1$-formula $\phi$, $\phi\equiv\psi$ for some $\LL_2$-formula $\psi$, and vice versa.

The logics of dependence mentioned above are defined as follows. Since in this paper we consider the logics from a semantical point of view, using the team semantics, we define their finitary consequence relations semantically. 

\begin{dfn}[Consequence relations for logics of dependence]\label{def_logics}
For a logic $\LL \in \{\PDo,\Inql,\PTo \}$, formulas $\phi$ and finite sets of formulas $\Gamma$, $\Gamma\vdash_\LL \phi$ if and only if $\phi$ and all formulas in $\Gamma$ are in \lang{\LL} and $\Gamma \models \phi$. 
$\phi$ is \emph{valid in \LL}, or a {\em theorem} of \LL, if $\vdash_\LL\phi$, which is short for $\emptyset \vdash_\LL \phi$. Thus theorems of \PDo and \Inql are the restrictions of the theorems of \PTo to \lang{\PDo} and \lang{\Inql}, respectively. 
\end{dfn}

Because of the semantical definition of the consequence relation $\vdash_\LL$, soundness and completeness with respect to the team semantics trivially holds. We will see, however, that there do exist genuine syntactic characterizations of dependence logics, as given in Theorem~\ref{pid_completeness} and the comments thereafter. Since in this paper the methods are purely semantical, the semantically defined consequence relations suffice for our aims. 


We write $\phi(p_1,\dots,p_n)$ if the propositional variables occurring in $\phi$ are among  $p_1,\dots,p_n$. Given a set $V$ of propositional variables, a \emph{valuation on $V$} is a function $v:V\to \{0,1\}$, and a \emph{team on $V$} is a set of valuations on $V$. 

\begin{thm}\label{basic_prop_pt}
Let $\phi(p_1,\dots,p_n)$ be a formula and $\Gamma$ a set of formulas in \lang{\PTo}, and $X$ and $Y$ two teams. Then the following holds. 
\begin{description}
\item[(Locality)] If $\{v\upharpoonright \{p_1,\dots,p_n\}: v\in X\}=\{v\upharpoonright \{p_1,\dots,p_n\}:v\in Y\}$, then
\[X\models\phi\iff Y\models\phi.\] 
\item[(Downwards Closure Property)] If $X\models\phi$ and $Y\subseteq X$, then $Y\models\phi$.
\item[(Empty Team Property)] $\emptyset\models\phi$.
\item[(Deduction Theorem)] $\Ga,\phi \models \psi$ if and only if $\Ga \models \phi \imp \psi$. 
\item[(Compactness Theorem)] If $\Gamma\models\phi$, then there exists a finite set $\Delta\subseteq \Gamma$ such that $\Delta\models\phi$.
\end{description}
\end{thm}

Given a formula $\phi$ and a finite set $\{\phi_i\mid i\in I\}$ of formulas we introduce a meta-symbol $\bigsqcup$ and use 
$\phi \bigsqcup_{i\in I} \phi_i$ as an abbreviation for the statement: For all teams $X$: 
$X\models \phi_i$ implies $X \models \phi$ for all $i\in I$, and $X \models \phi$ implies 
$X\models \phi_i$ for some $i\in I$.   

\begin{thm}[Disjunction property]\label{DP}
Let $\phi$ be a formula and $\{\phi_i\mid i\in I\}$ a finite set of formulas in \lang{\LL}. If $\phi\bigsqcup_{i\in I}\phi_i$ and $\models\phi$, then $\models\phi_i$ for some $i\in I$.
\end{thm}
\begin{proof}
Let $V=\{p_1,\dots,p_n\}$ be the set of propositional variables occurring in $\phi$ and $\{\phi_i\mid i\in I\}$. Since $\models\phi$, for the team $X=\{0,1\}^V$, we have that $X\models\phi$. It follows from $\phi \bigsqcup_{i\in I} \phi_i$ that $X\models\phi_i$ for some $i\in I$. Noting that every team $Y$ on $V$ is a subset of $X$, by the downwards closure property we obtain that $Y\models\phi_i$, which implies $\models\phi_i$ by locality.
\end{proof}



A formula of \PTo is said to be \emph{classical} if it does not contain any dependence atoms or intuitionistic disjunction. Classical formulas $\phi$ of \PTo are \emph{flat}, that is,
\[X\models\phi\iff \forall v\in X,~\{v\}\models\phi\]
holds for all teams $X$. 
The following lemma shows that classical tautologies of \PTo are exactly the tautologies of classical propositional logic.


\begin{lem}\label{flatv=classicalv}
For any classical formula $\phi$ in \lang{\PTo}, identifying  tensor  disjunction with classical disjunction of \CPC, we have that
\(\models_\CPC\phi\iff\models_\PTo\phi.\)
\end{lem}
\begin{proof}
An easy inductive proof shows that $v\models_{\CPC}\phi\iff\{v\}\models_\PTo\phi$ for all valuations $v$ and all classical formulas $\phi$.
\end{proof}

Having the same syntax as intuitionistic logic, the logic \Inql has a close relationship with intermediate logics between \ND and \ML. In \cite{InquiLog}, a Hilbert-style deductive system for \Inql is given. The axioms of this system will play a role in this paper, so we present the system in detail as follows. 

\begin{thm}[\cite{InquiLog}]\label{pid_completeness}
\Inql is sound and strongly complete with respect to the following Hilbert-style deductive system:
\begin{description}
\item[Axioms:] \
\begin{enumerate}
\item all theorems of \/\IPC 
\item $\neg\neg p\to p$ for all $p\in\rm{Prop}$
\item   all substitution instances of $\mathrm{ND}_k$ for all $k\in\mathbb{N}$:
\[(\mathrm{ND}_k)~~~~~~\big(\neg \phi\to\bigvee_{1\leq i\leq k}\neg \psi_i\big)\to \bigvee_{1\leq i\leq k}(\neg \phi\to\neg \psi_i).\]
\end{enumerate}
\item[Rule:] \
\begin{description}
\item[ Modus Ponens:] \AxiomC{$\phi\to \psi$} \AxiomC{$\psi$}\BinaryInfC{$\psi$} \DisplayProof ~($\mathsf{MP}$)
\end{description} 
\end{description}
\end{thm}


\begin{rmk}\label{remarkpid}
\Inql extended with dependence atoms is called \emph{propositional intuitionistic dependence logic} (\PIDo) in the literature (see e.g., \cite{Yang_dissertation,VY_PD}).
 As noted in \cite{Yang_dissertation,VY_PD}, \PIDo and \Inql have the same expressive power, as dependence atoms are definable in \Inql:
 \begin{equation}\label{dep_df}
\dep(p_1,\dots,p_n,q)\equiv (p_1\vee\neg p_1)\wedge\dots\wedge(p_n\vee\neg p_n)\to (q\vee\neg q).
\end{equation}
  Adding an axiom that corresponds to the above equivalence to the deductive system of \Inql,  one obtains a complete axiomatization for \PIDo. For simplicity, we will not discuss the logic \PIDo in this paper, but we remark that results obtained in this paper can be easily generalized to \PIDo.
\end{rmk}

The logic \PDo was first axiomatized by a natural deduction system in \cite{Yang_dissertation,VY_PD}, and a Hilbert-style axiomatization and a labelled tableau calculus for \PDo can be found in \cite{SanoVirtema2014}. Based on these, a natural deductive system for the   fragment of \PTo without dependence atoms was given in \cite{Ciardelli2015}. Adding to the deductive system in   \cite{Ciardelli2015} obvious rules for dependence atom that correspond to the equivalence in (\ref{dep_df}), 
one easily obtains a complete natural deductive system for full \PTo. Interested readers are referred to the literature given for the exact definitions of the deductive systems. Throughout this paper, we take for granted the strong completeness theorem for these logics.

It is important to note that the deductive systems for \PDo, \Inql and \PTo do \emph{not} admit \emph{uniform substitution}. 
Here substitutions, a crucial notion in this paper, are defined as follows. The definition is sufficiently general to apply to both propositional logics of dependence and intermediate logics that we consider later in the paper. 

\begin{dfn}[Substitution]\label{def_subst}
A \emph{substitution} of a propositional logic or theory \LL is a mapping $\sigma$ from the set of all formulas in \lang{\LL} to the set of all formulas in \lang{\LL}, that commutes with the connectives and atoms. 
\end{dfn}

\begin{dfn}
Let $\vdash_\LL$ be a consequence relation of a logic or  theory \LL. A substitution $\sigma$ is called a \emph{$\vdash_\LL$-substitution} if $\vdash_\LL$  is \emph{closed under} $\sigma$, i.e., for all formulas 
$\phi,\psi$ in \lang{\LL},
\[\phi\vdash_\LL\psi\Longrightarrow \sigma(\phi)\vdash_\LL\sigma(\psi).\] 
If $\vdash_\LL$ is closed under all substitutions, then we say that $\vdash_\LL$ is \emph{structural}.
\end{dfn}

The consequence relations of the logics \PDo, \Inql and \PTo are not structural, because, for example,  $p\sor p\vdash_\PDo p$ and $\vdash_\Inql \neg\neg p \imp p$, but $\dep(p)\sor\dep(p)\nvdash_\PDo\dep(p)$ and $\nvdash_\Inql \neg\neg(p\vee\neg p) \imp p\vee\neg p$.

\subsection{Normal forms}\label{secnf}

In this section, we recall from \cite{InquiLog} and \cite{VY_PD} the disjunctive normal forms for formulas of \PDo,  \Inql and \PTo. 
These normal forms, reminiscent of the disjunctive normal form in classical logic, play an important role in the main proofs of this paper and are defined as follows.

Fix $V=\{p_1,\dots,p_n\}$. Let $X$ be a nonempty team on $V$. For each of the logics \PDo,  \Inql and \PTo, we define a formula $\Theta_X$ as follows:
\begin{numcases}{\Theta_X:=}
\displaystyle\bigsor_{v\in X}(p_{1}^{v(p_1)}\wedge\dots\wedge p_{n}^{v(p_n)})&\text{ for }\PDo, \label{Theta_X_df1}\\
\displaystyle\neg\neg\bigbor_{v\in X}(p_{1}^{v(p_1)}\wedge\dots\wedge p_{n}^{v(p_n)})&\text{ for } \Inql,\ \PTo, \label{Theta_X_df2}
\end{numcases}


where $p^1:=p$ and $p^0:=\neg p$ and we stipulate that $\Theta_\emptyset:=\bot$. The reader can verify readily that the above two formulas are semantically equivalent. This is why we decide to be sloppy here and use the same notation $\Theta_X$ to stand for two syntactically different formulas. We tacitly assume that $\Theta_X$ is given by \eqref{Theta_X_df1} in the context of $\PDo$  and by \eqref{Theta_X_df2} in the context of  \Inql. For \PTo we could as well have chosen  \eqref{Theta_X_df1} as the definition of $\Theta_X$, as both defining formulas belong to \lang{\PTo} and are equivalent. 

With respect to the domain $V$, the formula $\Theta_X$ defines the team $X$ (module subteams), as stated in the following lemma, whose proof is left to the reader or see \cite{VY_PD}.

\begin{lem}\label{theta_prop}
Let $X$ and $Y$ be teams on $V$. For  the logics \PDo, \Inql and \PTo, we have that
\(Y\models\Theta_X\iff Y\subseteq X.\)
\end{lem}


The set $\llbracket \phi\rrbracket=\{X\subseteq \{0,1\}^V: X\models \phi\}$ is nonempty (as $\emptyset\in \llbracket \phi\rrbracket$) and \emph{downwards closed}, i.e., $Y\subseteq X\in \llbracket \phi\rrbracket\Longrightarrow Y\in \llbracket \phi\rrbracket$. We say that a propositional logic \LL of dependence is \emph{expressively complete}, if every nonempty downwards closed collection $\mathcal{K}$ of teams on $V$ is definable by a formula $\phi$ in \lang{\LL}, i.e., $\mathcal{K}=\llbracket \phi\rrbracket$.


\begin{thm}[\cite{InquiLog}\cite{VY_PD}]\label{NF_lm}
\begin{description}
\item[(i)] All of the logics \PTo, \PDo and \Inql are expressively complete and have the same expressive power.
\item[(ii) (Normal Forms)] Let $\phi(p_1,\dots,p_n)$ be a  formula in \lang{\PTo} or \lang{\PDo} or \lang{\Inql}. There exists a finite collection $\{X_i\mid i\in I\}$ of teams on $V$ such that $\phi\bigsqcup_{i\in I}\Theta_{X_i}$. In particular, $\phi\equiv\bigvee_{i\in I}\Theta_{X_i}$ holds for $\PTo$ and $\Inql$.
\end{description}
\end{thm}
\begin{proof}
We only give a proof sketch.  For (i), let $\mathcal{K}$  be a nonempty downwards closed collection of teams on $V$. The formula $\bigvee_{X\in \mathcal{K}}\Theta_X$ in \lang{\PTo} or \lang{\Inql} satisfies $\mathcal{K}=\llbracket\bigvee_{X\in \mathcal{K}}\Theta_X\rrbracket$ by \Cref{theta_prop}. The proof for the logic \PDo follows from a different argument; we refer the reader to \cite{VY_PD} for details.	

For every formula $\phi$, the set $\llbracket \phi\rrbracket$ is nonempty and downwards closed. Thus the item (ii) follows from the proof of item (i).
\end{proof}

\subsection{Intermediate logics}

There is a close relationship between logics of dependence and  intermediate theories (i.e., theories between intuitionistic and classical logic), as first formulated in \cite{InquiLog}. Here we describe this connection, and in the sections on projectivity and admissibility we will treat dependence logics and intermediate theories side by side.  

An \emph{intermediate theory} is a set \LL of formulas closed under modus ponens 
such that $\IPC\subseteq \LL\subseteq \CPC$. An \emph{intermediate logic} is an intermediate theory closed under uniform substitution. The intermediate logics that are most relevant in this paper are Maksimova's logic \ND, Kreisel-Putnam logic \KP and Medvedev's logic \ML (``the logic of finite problems''). It is well-known that $\ND\subseteq \KP\subseteq \ML$, and  \ML is the maximal intermediate logic extending \ND that has the disjunction property. 

We call a substitution $\sigma$ \emph{stable} in a logic \LL that has implication and negation in its language if $\sigma(p)$ is \emph{stable} in \LL, i.e., $\vdash_\LL\sigma(p)\leftrightarrow \neg\neg\sigma(p)$, for all $p\in\rm{Prop}$.
It is easy to verify that the substitution $(\cdot)^\neg$, defined as 
\(p^\neg=\neg p\text{ for all }p\in \rm{Prop},\)
is a stable substitution in all intermediate logics. For any intermediate logic \LL, define its \emph{negative variant} \LLn  as
\[\LLn=\{\phi\mid\phi^\neg\in\LL\}.\]

\begin{lem}[\cite{InquiLog}]\label{intermediate_neg}
Let \LL be an intermediate logic. 
\begin{description}
\item[(i)] \LLn is the smallest intermediate theory that contains \LL and 
$\neg\neg p\to p$ for every $p\in\rm{Prop}$. 
\item[(ii)] The consequence relation $\vdash_{\LLn}$ of \LLn is closed under stable substitutions. 
\item[(iii)] If \LL has the disjunction property, then so does \LLn.
\end{description}
\end{lem}

\begin{lem}
 Let \LL be an intermediate logic such that $\ND\subseteq\LL$. Every formula is provably equivalent to a formula of the form $\bigvee_{i\in I}\neg\phi_i$ in \LLn.
\end{lem}
\begin{proof}
The lemma follows essentially from \cite{InquiLog}.
Each formula $\neg\phi_i$ is a $\Theta_X$ formula as defined in (\ref{Theta_X_df2}) for some set $X$ of valuations, and the proof makes essential use of the axioms of \ND and \Cref{intermediate_neg}(i).
\end{proof}

It was shown in \cite{InquiLog}  that the negative variants of all of the intermediate logics between \ND and \ML (including \KP) are identical. 
Propositional inquisitive logic \Inql is the negative variant of such logics. We state this and other properties of \Inql in the following theorem.

\begin{thm}[\cite{InquiLog}]\label{ndneg_NF}\label{inql=kpn}
\begin{description}
\item[(i)] For any intermediate logic \LL such that $\ND\subseteq \LL\subseteq \ML$, we have that $\Inql=\LL^\neg$.
\item[(ii)] \Inql has the disjunction property and its consequence relation $\vdash_\Inql$ is closed under stable substitutions.
\end{description}
\end{thm}

There are many intermediate logics, including \ND and \KP, for which not much is known about their 
admissible rules. In Theorem~\ref{strcpl_neg_int} we show that the negative fragment of 
intermediate logics between \ND and \ML is structurally complete with respect to stable substitutions. 
Although we cannot immediately draw conclusions from this about the admissibility in the original logics, we hope that our results can be of help in the understanding of admissibility in these logics some day.  

\section{Extensions of the logics and substitutions}

\subsection{Extensions of the logics}
For intermediate logics and \Inql, all possible substitutions are well-defined, meaning that given a formula and a substitution in the language of the logic, applying that substitution to the formula results in a formula in that language. However, for the other logics of dependence that we consider in this paper (i.e., \PDo and \PTo),  substitution is \emph{not} well-defined.  A counter example is the formulas $\dep(p_1,\dots,p_n,q)$ and $\neg p$, for which the substitution instances $\dep(\sig p_1,\dots,\sig p_n,\sig q)$ and $\neg\sigma(p)$ only belongs to \lang{\PDo} or/and \lang{\PTo} if $\sig$ maps every propositional variable to a propositional variable. 

For the study of admissibility one has to isolate the (or a meaningful) set of \emph{well-defined} substitutions under which a consequence relation of a logic is \emph{closed}. For this purpose, in this section we expand the languages of the logics \PDo and \PTo so as to force flat substitutions to be well-defined, and we will show in the next section that these extensions are closed under flat substitutions.

\begin{dfn}\label{syntax_pdb}
The following grammars define well-formed formulas of the extended logics of dependence.
\begin{itemize}
\item 
The \emph{extended propositional downwards closed team logic} (\PT): 
\[
    \phi::= \,p\mid \bot \mid \top \mid\dep(\vv{\phi},\phi)\mid \neg\phi\mid \phi\wedge\phi\mid\phi\sor\phi\mid \phi\vee\phi\mid \phi\to\phi.
\] 
\item
The \emph{extended propositional dependence logic} (\PD): 
\[
    \phi::= \,p\mid \bot \mid \top \mid\dep(\vv{\alpha},\beta)\mid \neg\phi\mid\phi\wedge\phi\mid \phi\otimes\phi,
\] 
where $\vv{\alpha},\beta$ are flat formulas.
\end{itemize}
\end{dfn}

The extended logics have arbitrary negation as well-formed formulas. In the sequel we will give a semantics for the negation that is well conservative over the restricted negation in the original logics but not found in the literature. The extension \PT has dependence atoms with arbitrary arguments, while in the extension \PD we only allow dependence atoms with flat arguments. The restriction for \PD is made for technical simplicity that we discuss in the sequel, but as we consider flat substitutions only, this limitation does not affect the generality of the results in this paper. Generalized dependence atoms with flat arguments are also studied in the context of modal dependence logic, see \cite{EHMMVV2013}\cite{HLSV14}.



\vspace{8pt}

Below we define the  semantics of the new formulas. We first treat \PT and then \PD.

\begin{dfn}\label{extend_pt_df}
Let $\phi_1,\dots,\phi_n,\psi$ be arbitrary formulas of \PT. Define
\begin{description}
\item[(a)] $X\models \dep(\phi_1,\dots,\phi_n,\psi)$ iff $X\models\bigwedge_{i=1}^n(\phi_i\vee(\phi_i\to \bot))\to(\psi\vee(\psi\to\bot))$;\footnote{The authors would like to thank Ivano Ciardelli for suggesting this definition, see also \cite{Ciardelli2015}.}
\item[(b)] $X\models \neg\phi$ iff $X\models\phi\to\bot$ iff $\{v\}\not\models\phi$ for all $v\in X$.
\end{description}
\end{dfn}


In order for these definitions to be well-defined they have to agree with previously defined notions. For the dependence atom the observation in \eqref{dep_df} suffices. For negation, it suffices for 
\PT that $\neg\phi$ has been defined as a shorthand for $\phi\to\bot$, thus the semantics for negation as given in item (b) coincides with that in this logic. 

We turn to \PD. To define the semantics of the new formulas we need the following equivalence relation between valuations. 
Given a sequence $\vv{\phi}=\phi_1\dots\phi_n$ of formulas,  define an equivalence relation $\sim_{\vv{\phi}}$ on teams as follows:
\[
 u \eqv{\vv{\phi}} v\quad \text{ iff } \quad\A 1\leq i\leq n \, (\{u\}\models \phi_i \Ifff \{v\} \models \phi_i).
 \]

\begin{dfn}\label{extend_pd_df}
Define
\begin{description}
\item[(a)] for flat formulas $\alpha_1,\dots,\alpha_k,\beta$ of \PD, 
\begin{equation}\label{flat_dep}
X\models\dep(\vv{\alpha},\beta)\equiv_{\textsf{df}}\forall v,v'\in X(v\sim_{\vv{\alpha}}v'\Longrightarrow v\sim_\beta v');
\end{equation}
\item[(b)]
full negation in \PD  as $X\models \neg\phi$ iff $\{v\}\not\models\phi$ for all $v\in X$.
\end{description}
\end{dfn}

We have to show that the notions defined in Definition~\ref{extend_pd_df} 
are extensions of the corresponding notions for \PDo, and also special case of those of \PT. 
Obviously for the formula $\dep(\vv{p},q)$, the semantics given in item (a) coincides with the semantics given in \Cref{TS_PD}, and we leave it to the reader to check that it also coincides with \Cref{extend_pt_df}(a). 

The negation defined in item (b) deserves more comments. It is straightforward from the definition that $\neg\phi$ is always flat, and such defined negation coincides with that of \PT. In the literature of first-order dependence logic, negation is usually treated only \emph{syntactically}, in the sense that a negated formula $\neg\phi$ is defined to have the same semantics as the unique formula $\phi^{\sim}$ in \emph{negation normal form} obtained  by exhaustively applying the De Morgan's laws and some other syntactic rewrite rules. The corresponding syntactic rewrite rules for propositional dependence logic are as follows:
\begin{equation}\label{neg_df}
\begin{array}{rclcrclcrcl}
 p^\sim&\mapsto& \neg p &~~~~~~~~& \top^\sim&\mapsto&\bot&~~~~~~~~& (\phi\wedge\psi)^\sim&\mapsto&\phi^\sim\sor\psi^\sim\\
(\neg p)^\sim&\mapsto& p  &&\bot^\sim&\mapsto&\top& &(\phi\sor\psi)^\sim&\mapsto&\phi^\sim\wedge\psi^\sim \\
\dep(\vv{\phi},\psi)^\sim &\mapsto&\bot 
 \end{array}
\end{equation}
It is easy to see that the syntactic rewrite procedure for a negated formula $\neg\phi$ of \PD defined as above always terminates on a unique dependence atom-free formula $\phi^{\sim}$ in negation normal form in \lang{\PDo}. 

When applying the syntactic negation, special attention needs to be paid to double negations of dependence atoms, i.e., formulas of the form $\neg\neg\dep(\vv{a},b)$, where the variables $\vv{a},b$ are  first-order or propositional. Following Hintikka's game-theoretic perspective of logic (see, e.g., \cite{Hintikka98book}), the negation in logics of dependence is usually treated as a connective upon reading which the two players in the corresponding semantic game swap their roles. This way $\neg\neg\dep(\vv{a},b)$ should have the same meaning as $\dep(\vv{a},b)$, however, this reading is not consistent with the syntactic rewrite rules as in (\ref{neg_df}) (see e.g., \cite{Ville_thesis} for further discussions). To avoid ambiguity, most literature of logics of dependence does not allow double negation to occur in front of dependence atoms. In this paper, in the extended logic \PD we do include double negated dependence atoms as well-formed formulas, but as we do not take the game-theoretic approach to propositional logics of dependence, the semantics of double negated dependence atoms is computed simply according to \Cref{extend_pd_df}(b), namely, $\neg\neg\dep(\vv{p},q)$ is always semantically equivalent to $\top$ (noting that $\dep(\vv{p},q)$ is always true on singleton teams). Given such interpretation of the double negated dependence atoms, the negation defined in \Cref{extend_pd_df}(b) coincides with the syntactic negation given by the rewrite rules in (\ref{neg_df}), as we will show in the next lemma. 
However, on the other hand, in the context of first-order dependence logic, regardless how double negated dependence atoms are treated, the negation defined as in \Cref{extend_pd_df}(b)  does \emph{not} coincide with the syntactic negation given by the rewrite rules (rather, it corresponds to the defined connective $\sim\downarrow$ in Hodges \cite{Hodges1997a,Hodges1997b}). For instance, the reader who is familiar with the semantics of first-order dependence logic can easily verify that $M\not\models_{\{s\}}\forall x\dep(x)$ holds for all assignments $s$ on all models $M$, assuming that the domain of a model has at least two elements. Thus by \Cref{extend_pd_df}(b) $M\models_X\neg\forall x\dep(x)$ for all teams $X$ on all models $M$, namely $\neg\forall x\dep(x)\equiv \top$. However, by the syntactic rewrite rules, $(\forall x\dep(x))^\sim=\exists x  (\dep(x))^\sim=\exists x\bot$.

\begin{lem}\label{neg_syn_sem}
For any formula $\phi$ in \lang{\PD}, we have that $\neg\phi\equiv\phi^\sim$.
\end{lem}
\begin{proof}
We prove by induction on $\phi$ that $X\models\neg\phi\iff X\models\phi^\sim$ for all teams $X$.

The case $\phi=p$ or $\bot$ or $\top$ is easy. If $\phi=\neg p$, then $\phi^\sim=p$ and we have that $X\models \neg\neg p\iff \forall v\in X:\{v\}\not\models \neg p\iff \forall v\in X:\{v\}\models p\iff X\models p$.

If $\phi=\dep(\vv{p},q)$, then $\phi^\sim=\bot$ and we have that $X\models \neg\dep(\vv{p},q)\iff \forall v\in X: \{v\}\not\models \dep(\vv{p},q)\iff X=\emptyset\iff X\models \bot$.

If $\phi=\psi\wedge\chi$, then $\phi^\sim=\psi^\sim\otimes \chi^\sim$ and we have  that
\begin{align*}
X\models \neg(\psi\wedge\chi)\iff& \forall v\in X: \{v\}\not\models \psi\wedge\chi\\
\iff& \exists Y,Z\subseteq X\text{ s.t. }(\forall v\in Y: \{v\}\not\models\psi)\text{ and }(\forall u\in Z: \{u\}\not\models\chi)\\
\iff& \exists Y,Z\subseteq X\text{ s.t. }Y\models \neg \psi\text{ and }Z\models \neg\chi\\
\iff& \exists Y,Z\subseteq X\text{ s.t. }Y\models \psi^\sim\text{ and }Z\models\chi^\sim\\
&\text{(by the induction hypothesis)}\\
\iff& X\models \psi^\sim\otimes\chi^\sim.
\end{align*}

If $\phi=\psi\otimes\chi$, then $\phi^\sim=\psi^\sim\wedge \chi^\sim$ and we have by the induction hypothesis that
\begin{align*}
X\models \neg(\psi\otimes\chi)\iff& \forall v\in X: \{v\}\not\models \psi\otimes\chi\\
\iff& \forall v\in X: \{v\}\not\models \psi\text{ and }\{v\}\not\models\chi\\
\iff& X\models \neg\psi\text{ and }X\models\neg\chi\\
\iff& X\models \psi^\sim\wedge\chi^\sim\text{ (by the induction hypothesis)}.
\end{align*}
\end{proof}

It is evident from \Cref{extend_pd_df}(b) that the full negation of \PD is a \emph{semantic connective}. An $k$-ary connective $\cnt$ is called a \emph{semantic connective}, if 
\begin{center}
$\phi_1\equiv\psi_1,\,\dots,\,\phi_k\equiv\psi_k\,\Longrightarrow\,\cnt(\phi_1,\dots,\phi_k)\equiv\cnt(\psi_1,\dots,\psi_k)$.
\end{center}
 \Cref{neg_syn_sem} states that the semantical negation of \PD defined in \Cref{extend_pd_df}(b)  and the syntactic negation given by (\ref{neg_df})  coincide. It is worth emphasizing that in contrast to  \PD and other familiar logics with negation, the syntactic negation of first-order dependence logic is \emph{not} a semantic connective (regardless how double negated dependence atoms are treated), as shown by Burgess \cite{Burgess_negation_03} and V\"{a}\"{a}n\"{a}nen and Kontinen \cite{Negation_D_KV}. For an illustration,  $\forall x\dep(x)\equiv\forall x\forall y(x=y)$, whereas by the syntactic rewrite rules $(\forall x\dep(x))^\sim=\exists x\bot\not\equiv \exists x\exists y(x\neq y)=(\forall x\forall y(x=y))^\sim$.  %
 
The logics \PDo and \PTo  are expressively complete, therefore their extensions have the same expressive power as the original ones. Thus it is straightforward to verify that \Cref{basic_prop_pt} and \Cref{NF_lm} hold also for the extended logics \PD and \PT. 
 One can easily extend the deductive systems of the original logics by adding characterization rules for the negation and generalized dependence atoms and prove the sound and completeness theorems for the extensions. To characterize the negation, to the deductive systems of \PTo and \PDo one adds the obvious rules that characterize the equivalence between $\neg\phi$ and $\phi\to\bot$, and the obvious rules that characterize the rewrite rules in (\ref{neg_df}), respectively. To characterize generalized dependence atoms, to the deductive system of \PTo one adds obvious rules that correspond to  the equivalence in \Cref{extend_pt_df}(a). For \PD, following the idea in \cite{VY_PD} one generalizes the rules for dependence atoms in the deductive system of \PDo according to the equivalence in \Cref{extend_pt_df}(a) in an obvious way. To prove the completeness theorem for such obtained system of \PD,  one observes that whenever $\phi_1,\dots,\phi_n,\psi$ are flat,
\begin{equation}\label{gdep_nf}
\dep(\phi_1,\dots,\phi_n,\psi)\equiv\bigvee_{f\in \{0,1\}^X}\bigsor_{v\in X}(\phi_{1}^{v(\phi_1)}\wedge\dots\wedge \phi_{n}^{v(\phi_n)}\wedge\psi^{f(v)})
\end{equation}
holds, where $X=\{0,1\}^{\{\phi_1,\dots,\phi_n\}}$, $\phi_i^1=\phi_i$ and $\phi_i^0=\neg\phi_i$, and modifies the definition of a \emph{realization} of a generalised dependence atom accordingly. Note that if the arguments $\phi_1,\dots,\phi_n,\psi$ of a generalized dependence atom are not assumed to be flat, Equation (\ref{gdep_nf}) will no longer hold, and we do not see at this moment how to  obtain a complete axiomatization of the extended logic also in the general case. But since the notion of admissibility we study in this paper concerns \emph{theoremhood} of our logics only, and we intensionally defined the consequence relations of our logics in a semantic manner (see \Cref{def_logics}), this obstacle in the axiomatization of the extended logic is not essential for the main results of this paper. In view of this, for simplicity in \PD we only allow generalized dependence atoms with flat arguments.

\subsection{Closure under flat substitutions}

The consequence relations of the logics \PD,  \Inql, and \PT are not structural. In this section we prove, however, that the consequence relations of these logics are closed under \emph{flat substitutions}, i.e., substitutions $\sigma$ such that $\sigma$ maps propositional variables to flat formulas. 
To this end, we define the following translation on teams. 
For any valuation $v$ and any substitution $\sig$, define a valuation $v_\sig$ as
\[
 v_\sig (p) =  
  \left\{
   \begin{array}{ll}
    1 & \text{if $\{v\}\models \sig(p)$;} \\
    0 & \text{if $\{v\}\not\models \sig(p)$.} \\
   \end{array}
  \right.
\]
For any team $X$, we define
\(
 X_\sig = \{v_\sig \mid v \in X \}.
\)
Given a team $Y \subseteq X_\sig$, let $Y_X^{\sig}$ denote the set $\{v \in X \mid v_\sig \in Y\}$. Clearly $Y^\sig_X \subseteq X$ and $(Y_X^{\sig})_\sig = Y$. 

\begin{lem}
 \label{lemsigmaoperation} 
 Let $\LL\in \{\PD,  \Inql, \PT\}$.
For all formulas $\phi$  and all flat substitutions $\sig$ in \lang{\LL}, 
\[
 X \models \sig(\varphi) \iff X_\sig \models \varphi.
\]
\end{lem}
\begin{proof}

We prove this lemma for all three logics at the same time by induction on the complexity of $\varphi$, where we use the following complexity $c(\phi)$ on formulas in \lang{\PT}. The use of the complicated clause for the dependence atom will become clear in the proof below. 
\[
 \begin{array}{rcll}
  c(p) & = & 0           & \text{$p$ a propositional variable} \\
  c(\bot) & = & 0           &  \\
  c(\top) & = & 0           &  \\
  c(\neg \phi) & = & c(\phi)+1 &\\
  c(\phi \circ \psi) & = & c(\phi)+c(\psi)+1 & \circ \in \{\wedge,\rightarrow,\sor \} \\ 
  c(\depp{\phi}{\psi}) & = & \big(\sum_{i=1}^n(2c(\phi_i)+4)\big) +2c(\psi) + 4 & \text{where } \vec{\phi} = \phi_1,\dots, \phi_n. 
 \end{array}
\]


The cases $\phi=\bot$ and $\phi=\top$ are trivial. Since $\sig(p)$ is flat, the following equivalences hold:
\[
  X \models \sig(p) \iff \A v \in X (\{v\}\models \sig(p)) \iff 
  \A v_\sig \in X_\sig (\{v_\sig\} \models p) \iff X_\sig \models p.
\]
Thus the case $\phi=p$ is proved. 

Case $\phi=\depp{\theta}{\psi}$.  For \PT, from \Cref{extend_pt_df}(a) we know that $\phi$ is semantically equivalent to a formula  in its language whose subformulas are of lower complexity, thus this case is reduced to the other cases. However for \PD, the equivalent formula given by \Cref{extend_pt_df}(a) is not in its language, neither does the equivalent formula given by Equation (\ref{gdep_nf}). Since \PDo is expressively complete, there indeed exists a formula $\phi'$ in the language of \PDo that is equivalent to $\phi$. However, this translation is not done in a compositional manner, neither in an inductive manner (see \Cref{NF_lm}). We therefore cannot reduce this case to the other cases for \PD,
as the reduction would assume
\[\phi\equiv\phi'\Longrightarrow \sigma(\phi)\equiv\sigma(\phi'),\]
a fact that we  establish only in \Cref{cons_flatsub}. To avoid such a circular argument, we now proceed to prove this case for \PD directly, using the equivalent semantics given in \Cref{extend_pt_df}(a) and assuming that $\vec{\theta}$ and $\psi$ are flat. 

For the direction ``$\Longrightarrow$'', assume $X \models \dep(\sigma(\vec{\theta}),\sigma(\psi))$ and $Y\models \bigwedge_{i=1}^n(\theta_i\vee\neg\theta_i)$ for some $Y\subseteq X_\sigma$. As $(Y_X^\sigma)_\sigma=Y$ and $c(\bigwedge_{i=1}^n(\theta_i\vee\neg\theta_i))<\Sigma_{i=1}^n(2c(\theta_i)+4)<c(\depp{\theta}{\chi})$, by the induction hypothesis, we obtain that $Y^\sigma_X\models \bigwedge_{i=1}^n(\sigma(\theta_i)\vee\neg\sigma(\theta_i))$. Clearly $Y_X^\sigma\subseteq X$, thus the assumption implies that $Y_X^\sigma\models\sigma(\psi)\vee\neg\sigma(\psi)$, which by the induction hypothesis again gives the desired $(Y_X^\sigma)_\sigma\models\psi\vee\neg\psi$, because $c(\psi\vee\neg\psi)=2c(\psi)+3<c( \depp{\theta}{\psi})$.
The other direction ``$\Longleftarrow$'' is symmetric, using $Y_\sigma\subseteq X_\sigma$ for any $X,Y$ with $Y\subseteq X$.

The cases that $\phi=\psi\wedge\chi$ and $\phi=\psi\vee\chi$ follow immediately from the induction hypothesis. 

Case $\phi=\psi \sor \chi$. 
We first prove the direction ``$\Longrightarrow$''. Assume $X \models \sig(\varphi)$ and consider $Y,Z \subseteq X$ such that $X = Y \cup Z$ and $Y \models \sig(\psi)$ and $Z \models \sig(\chi)$. Using that $Y_\sig \cup Z_\sig = X_\sig$, this implies $X_\sig\models \psi \sor \chi$ by the induction hypothesis. 

For the direction ``$\Longleftarrow$'', assume $X_\sig \models \varphi$ and consider $Y,Z \subseteq X_\sig$ such that $X_\sig = Y \cup Z$, $Y \models \psi$ and $Z \models \chi$. Thus 
$Y_X^{\sig}\models \sig(\psi)$ and $Z_X^{\sig}\models \sig(\chi)$ by the induction hypothesis. Since $X = Y_X^{\sig} \cup Z_X^{\sig}$, this implies $X \models \sig(\psi)\sor \sig(\chi)$, as required.

Case $\phi=\psi\imp \chi$. We first prove the direction ``$\Longrightarrow$''. Assume $X \models \sig(\varphi)$ and consider 
$Y \subseteq X_\sig$ such that $Y \models \psi$. As $(Y_X^{\sig})_\sig = Y$, $Y_X^{\sig} \models \sig(\psi)$ follows by the induction hypothesis. And as $Y^\sig_X \subseteq X$, this implies $Y_X^{\sig} \models \sig(\chi)$. Hence $Y\models \chi$ by the induction hypothesis, as required.
The direction ``$\Longleftarrow$'' is similar.

Case $\phi=\neg\psi$. It follows from the induction hypothesis that $X\models \neg \sigma(\psi)\iff \forall v\in X:\{v\}\not\models \sigma(\psi)\iff \forall v\in X: \{v_\sigma\}\not\models \psi\iff X_\sigma\models\neg\psi$.
\end{proof}

\begin{lem}\label{flat_closure}
The set of flat formulas in \lang{\PT} is closed under flat substitutions, i.e., whenever $\phi$ is a flat formula and $\sigma$ is a flat substitution, $\sigma(\phi)$ is flat too.
\end{lem}
\begin{proof}
Suppose $X$ is a team such that for all $v \in X$, $\{v\} \models \sig(\phi)$. To show that $X \models \sig(\phi)$, by 
\Cref{lemsigmaoperation}, it suffices to show that $X_\sig \models \phi$. As $\phi$ is flat, 
we therefore have to show that $\{v_\sig \}\models \phi$ for all $s \in X$. Again by Lemma~\ref{lemsigmaoperation} it suffices to show that $\{v\} \models \sig(\phi)$ for all $v \in X$. But that is what we assumed, so we are done. 
\end{proof}

As a consequence of the above lemma, for every generalized dependence atom $\dep(\phi_1,\dots,\phi_n,\psi)$ in \lang{\PD}, where $\phi_1,\dots,\phi_n,\psi$ are flat formulas, the resulting formula $\dep(\sigma(\phi_1),\dots,\sigma(\phi_n),\psi)$ under an arbitrary flat substitution $\sigma$ is still a well-formed formula in \lang{\PD}. This shows that flat substitutions are well-defined in \PD.

\begin{thm}\label{cons_flatsub}
The consequence relations of \PD, \Inql, and \PT are closed under flat substitutions. In particular, for all flat substitutions $\sigma$, we have that $\phi\equiv\psi$ implies $\sigma(\phi)\equiv\sigma(\psi)$.
\end{thm}
\begin{proof}
By the definition of the consequence relations, 
it suffices to prove that for all formulas $\phi$ and $\psi$, $\phi\models\psi\Longrightarrow\sigma(\phi)\models\sigma(\psi)$ holds for all flat substitutions $\sigma$.

Assume $\phi\models\psi$. We have that for any team $X$, any flat substitution $\sigma$,
\begin{align*}
X\models\sigma(\phi)&\Longrightarrow X_\sigma\models\phi~~\text{ (by \Cref{lemsigmaoperation})}\\
&\Longrightarrow X_\sigma\models\psi~~\text{ (by the assumption)}\\
&\Longrightarrow X\models\sigma(\psi)~~\text{ (by \Cref{lemsigmaoperation})}
\end{align*}
Hence $\sigma(\phi)\models\sigma(\psi)$.
\end{proof}

\section{Flat formulas and projective formulas}

Having proved that our logics are closed under flat substitutions  we work towards the proof of our main results by showing that flatness in these logics is nothing but projectivity, a key notion in the study of admissible rules. 

As the building blocks of the normal form of formulas in \lang{\PT}, the formulas $\Theta_X$, defined in Section~\ref{secnf}, turn out to be of particular interest. They actually serve as a syntactic characterization of flat formulas, as the following lemma shows. 


\begin{lem}\label{flat_charact}
Let $\phi(p_1,\dots,p_n)$ be a consistent formula in \lang{\PT}. The following are equivalent.
\begin{description}
\item[(i)] $\phi$ is flat;
\item[(ii)] $\phi\equiv\Theta_X$ for some nonempty team $X$ on $\{p_1,\dots,p_n\}$;
\item[(iii)] $\phi\equiv\neg\neg\phi$;
\item[(iv)] $\models\phi\sor\neg\phi$
\end{description}
\end{lem}
\begin{proof}
(ii)$\Rightarrow$(i) and (iii)$\Rightarrow$(i) follow from the fact that negated formulas are flat, and (i)$\Rightarrow$(iii) follows immediately from the definition of negation.

(i)$\Rightarrow$(ii):  In view of \Cref{theta_prop} and \Cref{NF_lm}, without loss of generality, we may assume that $\phi(p_1,\dots,p_n)=\bigvee_{i=1}^k\Theta_{X_i}$, where $\{X_1,\dots,X_k\}$ is a collection of some nonempty maximal (with respect to set inclusion) teams on $\{p_1,\dots,p_n\}$. Suppose $\phi$ is flat and $k>1$. For each $1\leq i<k$, pick $v_i\in X_i\setminus X_{i+1}$ and pick $v_k\in X_k\setminus X_1$. The maximality of the $X_i$'s guarantees that such $v_i$'s exist. Since $\{v_i\}\subseteq X_i$ and $\{v_1,\dots,v_k\}\nsubseteq X_i$ for all $1\leq i\leq k$, by \Cref{theta_prop},  $\{v_i\}\models \Theta_{X_i}$ and $\{v_1,\dots,v_k\}\not\models \Theta_{X_i}$ for all $1\leq i\leq k$, thereby $\{v_i\}\models \phi$ for all $1\leq i\leq k$ whereas $\{v_1,\dots,v_k\}\not\models \phi$. Hence we conclude that $k=1$ and $\phi=\Theta_{X_1}$, as required.

(i)$\Rightarrow$(iv): If $\phi$ is flat, to show (iv), it suffices to show $\{v\}\models\phi\otimes\neg\phi$, i.e., $\{v\}\models\phi$ or $\{v\}\models\neg\phi$, for all valuations $v$. But this is also obvious.

(iv)$\Rightarrow$(i): Suppose $\{v\}\models\phi$ for all valuations $v$ in a team $X$. Then $Y\not\models\neg\phi$ for all nonempty $Y\subseteq X$. Now, if $\models\phi\otimes\neg\phi$, then we must have that $X\models\phi$, which shows that  $\phi$ is flat.
\end{proof}

Since some of the logics we consider in this paper do not have implication in the language, and none of them is closed under uniform substitution, we modify the usual definition of projective formula.
\begin{dfn}[Projective formula]\label{def_proj}
Let \LL be a logic, and \Su a set of \LL-substitutions. A formula $\phi$ in \lang{\LL} is said to be  \emph{\Su-projective} in  \LL if there exists $\sigma\in \Su$  such that 
\begin{description}
\item[(a)] $\vdash_\LL\sigma(\phi)$,
\item[(b)] $\phi,\sigma(p)\vdash_\LL p$ and $\phi, p\vdash_\LL\sigma(p)$ for all propositional variables $p$. 
\end{description}
Such substitutions are called \emph{\Su-projective unifiers} for $\phi$ in \LL.
\end{dfn} 

Because of the Deduction Theorem (\Cref{basic_prop_pt}) of our logics that has implication in their languages the notion of projectivity can be formulated purely in terms of theoremhood. A standard inductive proof shows that the condition in \Cref{def_proj}(b) implies that $\phi,\sigma(\psi)\vdash_\LL \psi$ and $\phi,\psi\vdash_\LL\sigma(\psi)$ hold for all formulas $\phi$ and $\psi$ of our logics.

The proof of the following lemma uses what is known as {\em Prucnal's trick}, which consists of a method to prove projectivity via a connection between valuations and substitutions. 

\begin{lem}\label{theta_proj_pid}
Let $\LL\in\{\Inql,\PT\}$ and $X$ a nonempty set of teams on a finite set of propositional variables. The formula $\Theta_X$ in \lang{\LL} (defined by Equation (\ref{Theta_X_df2})) is \Fl-projective in \LL , where \Fl is the class of all flat substitutions. 
\end{lem}
\begin{proof}

Put $\phi=\Theta_X$ and pick $v\in X$. View $\phi$ as a formula of \CPC, clearly we have $v(p_{1}^{v(p_1)}\wedge\dots\wedge p_{n}^{v(p_n)})=1$, thereby $v(\phi)=1$. Define a substitution $\sigma_v^{\phi}$  as follows:
\begin{equation}\label{prucnal_sigma}
\sigma_v^{\phi}(p)=\begin{cases}
\phi\wedge p,&\text{ if }v(p)=0;\\
\phi\to p,&\text{ if }v(p)=1.
\end{cases}
\end{equation}

Put $\sigma=\sigma_v^{\phi}$. Clearly, $\sigma(p)$ (in both cases) is classical, thus flat.

By a standard inductive argument, one proves that 
\begin{equation}\label{cpc_sigma}
\vdash_\CPC\sigma(\psi) \iff v(\psi)=1
\end{equation}
 for all subformulas $\psi$ of $\phi$. Now, as $v(\phi)=1$, we obtain $\vdash_\CPC\sigma(\phi)$. Since $\phi$ is a classical formula, by \Cref{flatv=classicalv} we derive $\vdash_{\LL}\sigma(\phi)$.  Moreover, it follows from the definition of $\sigma$ that $\vdash_{\LL}\phi\to(\sigma(p)\leftrightarrow   p)$ holds for all  $p\in \rm{Prop}$. 
Hence we conclude that $\phi$ is \Fl-projective in \LL.
\end{proof}

It is known that negated formulas $\neg\phi$ are projective in every intermediate logic \LL, it follows, for example, from Ghilardi's characterization in \cite{ghilardi99}. Here we prove that the same holds for the negative variants of intermediate logics and that the projective unifiers involved are moreover stable. 

\begin{lem}\label{neg_proj}
Let \LL be an intermediate logic. Every consistent formula $\neg\phi$ is \St-projective in \LLn, where \St is the class of all stable substitutions.
\end{lem}
\begin{proof}
Take a valuation $v$ such that $v(\neg\phi)=1$. Define a substitution $\sigma^{\neg\phi}_v$ for $\neg\phi$  in exactly the same way as in (\ref{prucnal_sigma}) of the preceding lemma. Put $\sigma=\sigma^{\neg\phi}_v$. The definition of $\sigma$ guarantees that $\vdash_{\LL}\phi\to(\sigma(p)\leftrightarrow  p)$ holds for all $p\in\rm{Prop}$. By (\ref{cpc_sigma}) and Glivenko's Theorem (see e.g. Theorem 2.47 in \cite{ZhaCha_ml}), we obtain that $\vdash_\LL\neg\sigma(\phi)$. Hence we have proved that $\neg\phi$ is projective in \LL. Now, by \Cref{intermediate_neg} $\LL\subseteq \LLn$, thus $\neg\phi$ is projective also in \LLn. 


It remains to check that the $\sigma$ defined as above is a stable substitution in \LLn, i.e., $\vdash_\LLn\sigma(p)\leftrightarrow \neg\neg\sigma(p)$ for all $p\in\rm{Prop}$. If $v(p)=0$, then by the definition, we have that $\sigma(p)=\neg\phi\wedge p$. Since $\vdash_\LLn\neg\neg p\to p$ (by \Cref{intermediate_neg}), we have that
\[\neg\neg\sigma(p)=\neg\neg(\neg\phi\wedge p)\dashv\vdash\neg\neg\neg\phi\wedge \neg\neg p\dashv\vdash\neg\phi\wedge p=\sigma(p),\]
as required. If $v(p)=1$, then by the definition we have that $\sigma(p)=\neg\phi\to p$. Since $\vdash_\LLn\neg\neg p\to p$, we have that
\[\neg\neg\sigma(p)=\neg\neg(\neg\phi\to p)\dashv\vdash\neg\neg\neg\phi\to \neg\neg p\dashv\vdash\neg\phi\to p=\sigma(p),\]
as required. 
\end{proof}


\begin{lem}\label{theta_proj}
For any nonempty team $X$ on a set $\{p_1,\dots,p_n\}$ of propositional variables, the formula $\Theta_X$ in \lang{\PD} (defined by Equation (\ref{Theta_X_df1})) is \Fl-projective in \PD.
\end{lem}
\begin{proof}
This lemma is proved also using a similar argument to that of \Cref{theta_proj_pid}.
Put $\phi=\Theta_X$. Take an arbitrary $v\in X$. Clearly, $v(\phi)=1$ when $\phi$ is viewed as a formula of \CPC (hereafter in the proof, we identify tensor disjunction $\otimes$ with classical disjunction). Define a substitution $\sigma_v^{\phi}$  as follows:
\[\sigma_v^{\phi}(p)=\begin{cases}
\phi\wedge p,&\text{ if }v(p)=0;\\
\neg\phi\vee p,&\text{ if }v(p)=1.
\end{cases}
\]
Put $\sigma=\sigma_v^{\phi}$. Clearly, the formula $\sigma(p)$ (in both cases) is classical, thus flat.

As in the proof of  \Cref{theta_proj_pid}, we have that (\ref{cpc_sigma}) holds for all subformulas $\psi$ of $\phi$, thus $\vdash_\CPC\sigma(\phi)$. Now, since the formula $\sigma(\phi)$ is classical, we obtain by \Cref{flatv=classicalv} that $\vdash_{\PD}\sigma(\phi)$.

It remains to show that $\phi,\sigma(p)\vdash_{\LL} p$ and $\phi,p\vdash_\LL\sigma(p)$ for all $p\in\rm{Prop}$. If $v(p)=0$, then clearly $\phi,\phi\wedge p\vdash_{\PD} p$  and $\phi,p\vdash_{\PD}\phi\wedge p$. If $v(p)=1$, to see that $\phi,\neg\phi\sor p\vdash_{\PD}p$, if $X\models \phi\wedge(\neg\phi\sor p)$, then for all $v\in X$, we have that $\{v\}\models\phi\wedge(\neg\phi\sor p)$, which  implies that $\{v\}\models p$, thereby $X\models p$, as required. That $\phi,p\vdash_{\PD}\neg\phi\sor p$ follows easily from the fact that $p\vdash_{\PD}\neg\phi\sor p$.
\end{proof}



 \begin{lem}\label{flat_proj_charact}
Let $\LL\in\{ \PD,\Inql, \PT\}$, and $\phi$ a consistent formula in \lang{\LL}. The following are equivalent: 
\begin{description}
\item[(i)] $\phi\dashv\vdash \Theta_X$ for some nonempty $n$-team $X$;
\item[(ii)] $\phi$ is flat;
\item[(iii)] $\phi$ is \Fl-projective in \LL;
\end{description}
\end{lem}
 \begin{proof}
(ii)$\iff$(i)$\Longrightarrow$(iii) follows from \Cref{flat_charact,theta_proj_pid,theta_proj}.
Now, we show that (iii)$\Longrightarrow$(i). Suppose $\phi$ is \Fl-projective in \LL and $\sigma$ is a \Fl-projective unifier for $\phi$. Thus $\vdash_\LL\sigma(\phi)$, which implies $\models \sigma(\phi)$. 
By \Cref{DP}, 
this implies that there exists $1\leq i\leq k$ such that $\models\sigma(\Theta_{X_i})$. Since $\Theta_{X_i}$ is in \lang{\LL} and thus so is $\sigma(\Theta_{X_i})$, $\vdash_\LL\sigma(\Theta_{X_i})$ follows. On the other hand, we also have that $\phi,\sigma(\Theta_{X_i})\vdash_\LL\Theta_{X_i}$. It then follows that $\phi\vdash_\LL \Theta_{X_i}$. Hence $\phi\equiv \Theta_{X_i}$, which gives 
$\phi\dashv\vdash\Theta_{X_i}$.
\end{proof}

\begin{lem}\label{Lneg_proj_neg}
Let \LL be an intermediate logic such that $\ND\subseteq \LL$ and $\phi$ a consistent formula in \lang{\LL}. The following are equivalent:
\begin{description}
\item[(i)] $\vdash_\LLn\phi\leftrightarrow\neg\neg\phi$;
\item[(ii)] $\phi$ is \St-projective in \LLn;
\end{description}
\end{lem}
\begin{proof}
(i)$\Longrightarrow$(ii) follows from \Cref{neg_proj}. 
For (ii)$\Longrightarrow$(i), by \Cref{ndneg_NF},  in \LLn we have that $\phi\dashv\vdash\bigvee_{i\in I}\neg\phi_i$ for some formulas $\{\neg\phi_i\mid i\in I\}$. By a similar argument to that in the proof of ``(iii)$\Longrightarrow$(i)'' of \Cref{flat_proj_charact}, we obtain in \LLn that $\phi\dashv\vdash\neg\phi_i$ for some $i\in I$, which implies that $\vdash_{\LLn}\phi\leftrightarrow\neg\neg\phi$.
\end{proof}


\section{Structural completeness of the logics}

In this section we prove the main results of our paper, namely that the three propositional logics of dependence \PD, \Inql, and \PT are \Fl-structurally complete and that the negative variants of logics extending \ND are hereditarily \St-structurally complete. In both cases the proof of the fact is based on the existence, for every formula $\phi$, of certain \Su-projective formulas $\phi_i$ such that $\phi \bigsqcup_{i\in I} \phi_i$, where in the first case \Su consists of all flat substitutions and in the second case of all stable ones. As mentioned in Remark~\ref{remarkpid}, it is not hard to prove by the same methods that also the logic \PID  is \Fl-structurally complete, where  \PID is an extension of propositional intuitionistic dependence logic \PIDo in the same manner as \PT is an extension of \PTo. 

\begin{dfn}
Let \LL be a logic, and \Su a set of $\vdash_\LL$-substitutions. A rule $\phi/\psi$ of  \LL is said to be \emph{\Su-admissible}, in symbols $\phi\adm^{\Su}_\LL\psi$, 
if for all $\sigma\in \Su$, $\vdash_\LL\sigma(\phi)\,\Longrightarrow \,\vdash_\LL\sigma(\psi)$.

In case \Su is the set of all substitutions, we write $\adm$ for $\adm^{\Su}$, and such a rule is called an \emph{admissible rule}. 
\end{dfn}

\begin{dfn}
A logic \LL is said to be \emph{\Su-structurally complete} if every \Su-admissible rule of \LL is derivable in \LL, i.e., $\phi\adm^{\Su}_\LL\psi\iff\phi\vdash_\LL\psi$. In case \Su is the set of all substitutions and \LL is \Su-structurally complete, we say that \LL is \emph{structurally complete}.
\end{dfn}

Informally, a rule is admissible in a logic \LL if its addition to the logic does not change the theorems that are derivable. Clearly, if \Su is a set of \LL-substitutions, then $\phi\vdash_\LL\psi\Longrightarrow\phi\adm^{\Su}_\LL\psi$ for all formulas $\phi$ and $\psi$ in \lang{\LL}. In particular, by \Cref{cons_flatsub}, all derivable rules of \PD and \Inql are \Fl-admissible in the logics. 
A logic that is \Su-structurally complete has no nontrivial \Su-admissible rules: all such rules are derivable in the logic. Classical logic is structurally complete, but intuitionistic logic is not, as are many other intermediate logics. The well-known example showing that intuitionistic logic is not structurally complete uses \emph{Harrop's Rule}: 
\[
 \phi \imp \psi \vee \theta \adm_{\IPC} (\phi \imp \psi) \vee (\phi \imp \theta) \text{ and } 
 \phi \imp \psi \vee \theta \not\vdash_{\IPC} (\phi \imp \psi) \vee (\phi \imp \theta).
\]

Recall the definition of $\bigsqcup$ just below Lemma~\ref{basic_prop_pt}: 
$\phi \bigsqcup_i \phi_i$ holds if and only if $\phi_i \models \phi$ for all $i$, and for all teams $X$:  
$X\models \phi$ implies $X \models \phi_i$ for some $i$. 



\begin{lem}\label{lem_strcpl_proj} 
For any $\LL$ which is an intermediate theory or one of \PD, \Inql or \PT, and any set \Su of 
\LL-substitutions, if for every consistent formula $\phi$ in \lang{\LL} there exists a finite set 
$\{\phi_i\mid i\in I\}$ of \Su-projective formulas in \lang{\LL} such that 
$\phi \bigsqcup_{i\in I} \phi_i$, then \LL is \Su-structurally complete.  
\end{lem}
\begin{proof}
We show that every \Su-admissible rule $\phi\adm_\LL^\Su\psi$ of \LL is derivable, i.e., $\phi\vdash_\LL\psi$. If $\phi$ is inconsistent, then clearly $\phi\vdash_\LL\bot\vdash_\LL\psi$. Now assume that $\phi$ is consistent. By assumption there exists a finite set $\{\phi_i\mid i\in I\}$ of \Su-projective formulas  such that $\phi \bigsqcup_{i\in I} \phi_i$. Let $\sigma_{i} \in \Su$ be the projective unifier of $\phi_i$. 
Thus $\vdash_\LL\sigma_i(\phi_i)$. Hence $\vdash_\LL\sigma_i(\phi)$ for all $i \in I$. From $\phi\adm_\LL^\Su\psi$ we derive $\vdash_\LL\sigma_i(\psi)$ for each $i\in I$. 
Since $\sigma_i$ is a projective unifier for $\phi_i$, we have that $\phi_i,\sigma_i(\psi)\vdash_\LL\psi$. It follows that $\phi_i\vdash_\LL\psi$ for each $i\in I$. Therefore $\phi\vdash_\LL\psi$.
\end{proof}

\begin{thm}\label{maintheorem}
\PD, \Inql and \PT are \Fl-structurally complete. 
\end{thm}
\begin{proof}
By \Cref{NF_lm} for the extended logics, \Cref{flat_proj_charact,lem_strcpl_proj}.
\end{proof}

Let \LL be an intermediate theory/logic and \Su a set of $\vdash_\LL$-substitutions. 
We say that \LL is \Su-\emph{hereditarily structurally complete} if for
any intermediate theory $\LL'$ such that $\LL\subseteq \LL'$ and \Su is a set of $\vdash_{\LL'}$-substitutions, $\LL'$ is \Su-structurally complete. In case \Su is the class of all substitutions of \LL, then we say that \LL is \emph{hereditarily structurally complete}. It is known that none of \ND, \KP and \ML is hereditarily structurally complete.

\begin{thm}\label{strcpl_neg_int}
For any intermediate logic \LL such that $\ND\subseteq \LL$, its negative variant 
\LLn is \St-hereditarily structurally complete. In particular, $\ND^\neg$, $\KP^\neg$ and $\ML^\neg$ are \St-hereditarily structurally complete.
\end{thm}
\begin{proof}
By \Cref{ndneg_NF}, \Cref{Lneg_proj_neg,lem_strcpl_proj}.
\end{proof}

%
%

\section{Concluding remarks}

We have shown that the three propositional logics of dependence, \PD,  \PT, \Inql, are structurally complete with respect to flat substitutions and that the negative variant of every intermediate logic that is an extension of \ND is hereditarily structurally complete with respect to stable substitutions. In particular, $\ND^\neg$, $\KP^\neg$ and $\ML^\neg$ are \St-hereditarily structurally complete. The reason for this are the strong normal forms that hold in these logics or theories. In this aspect they resemble classical logic, with its disjunctive normal form, that is also hereditarily structurally complete. 

Apart from \cite{Migliolietc89} there has not been much research on admissibility on 
intermediate theories that are not intermediate logics, and for propositional logics of dependence 
the above results are the first of such kind. Thus, naturally, many questions remain open.  
We discuss several of them. 

Theorem~\ref{strcpl_neg_int} states that the negative variant of extensions of \ND are hereditarily structurally complete. It follows from results by Maxsimova and Prucnal that any structurally complete intermediate logic with the disjunction property contains \KP and is contained in \ML, and in \cite{wojtylak04}, which recaptures these results, it is moreover shown that \KP itself is not structurally complete. The same holds for \ND, since it is properly contained in \KP. 
One wonders whether the fact that the negative variant of \ND and \KP are structurally complete could shed some light on admissibility in the original logics.

In this paper the results on admissibility are with respect to sets of substitutions, such as the flat and the stable substitutions. There exist logics for which establishing whether admissibility has certain properties, such as decidability, seems hard. These problems are often considered only for admissibility with respect to all substitutions, but one could start with smaller sets of substitutions, which may be easier to deal with. And although certain properties, such as decidability of admissibility, do not transfer from a smaller set of substitutions to its extensions, understanding a restricted case may stil help understanding the general case. 

On a more abstract level, there are two definitions of admissibility in the literature that in most instances amount to the same notion. Although intuitively clear, the proper connection between the two is not completely straightforward \cite{iemhoff13,metcalfe12}. And it is mostly considered only for admissibility with respect to the set of all substitutions. It would be nice to see whether this connection can be generalized to admissibility with respect to any set of substitutions.  

The results obtained in this paper made essential use of the disjunctive normal form of formulas of propositional logics of dependence. It is known from the literature that modal dependence logic and propositional independence logic both have a similar disjunctive normal form \cite{Yang_dissertation,EHMMVV2013}. We conjecture that the argument in this paper may apply to these two logics and lead to similar results.

\begin{acknowledgements}
The authors would like to thank Ivano Ciardelli, Dick de Jongh and Jouko V\"a\"an\"anen for useful discussions on the topic of this paper.
\end{acknowledgements}

\bibliographystyle{spmpsci}      
\bibliography{fan}   


\end{document}